\documentclass[a4paper,dvips,12pt]{article}
\usepackage[latin1]{inputenc}
\usepackage{amsmath,amsthm,amsfonts,amssymb}
\usepackage{enumerate}
\usepackage[ps2pdf,colorlinks]{hyperref}
\usepackage{graphicx,psfrag}

\theoremstyle{plain}
\newtheorem{theorem}{Theorem}[section]
\newtheorem{lemma}[theorem]{Lemma}
\newtheorem{corollary}[theorem]{Corollary}
\newtheorem{proposition}[theorem]{Proposition}

\theoremstyle{definition}
\newtheorem{definition}[theorem]{Definition}

\newtheorem{remark}[theorem]{Remark}

\numberwithin{equation}{section}

\newcommand{\mailto}[1]{\href{mailto:#1}{\nolinkurl{#1}}}

\newcommand{\Z}{\mathbb{Z}}

\newcommand{\R}{\mathbb{R}}

\newcommand{\pr}{\operatorname{P}}

\newcommand{\abs}[1]{\left| #1 \right|}

\newcommand{\normtv}[1]{||#1||_{\rm TV}}

\begin{document}

\title{Geometric juggling with $q$-analogues}

\author{
 Alexander Engström\thanks{
 Postal address:
 Department of Mathematics and Systems Analysis,
 Aalto University, PO Box 11100, 00076 Aalto, Finland.
 Tel: +358 50 562 8761
 URL: \url{http://math.aalto.fi/en/people/alexander.engstrom} \quad
 Email: \protect\mailto{alexander.engstrom@aalto.fi}}
 \and
 Lasse Leskelä\thanks{
 Postal address:
 Department of Mathematics and Systems Analysis,
 Aalto University, PO Box 11100, 00076 Aalto, Finland.
 Tel: +358 50 434 9947
 URL: \url{http://math.aalto.fi/en/people/lasse.leskela} \quad
 Email: \protect\mailto{lasse.leskela@aalto.fi}}
 \and
 Harri Varpanen\thanks{
 Postal address:
 Department of Mathematics and Systems Analysis,
 Aalto University, PO Box 11100, 00076 Aalto, Finland.
 Tel: +358 50 593 6964.
 URL: \url{http://math.aalto.fi/en/people/harri.varpanen} \quad
 Email: \protect\mailto{harri.varpanen@aalto.fi}}
}
\date{\today}

\maketitle

\begin{abstract}
We derive a combinatorial stationary distribution for bounded juggling
patterns with a random, $q$-geometric throw distribution.
The dynamics are analyzed via rook placements on staircase
Ferrers boards, which leads to a stationary distribution
containing $q$-rook polynomial coefficients and $q$-Stirling numbers
of the second kind. We show that the stationary probabilities
of the bounded model can be uniformly approximated with the
stationary probabilities of a corresponding unbounded model.
This observation leads to new limit formulae for $q$-analogues.

\end{abstract}

\section{Introduction}
\label{sec:intro}

The mathematical modeling of juggling patterns started around 1980
when personal computers gained popularity and scholars envisioned juggling robots
and simulators. Shannon \cite{Shannon_1980} is
probably the first mathematical manuscript on juggling, with emphasis
on robotics.

The so-called siteswap juggling model was independently invented by several people in the early
1980s. It was first published by Magnusson and Tiemann \cite{Magnusson_Tiemann_1989}
and further developed by Buhler, Eisenbud, Graham and Wright \cite{Buhler_et_al_1994}. 
Siteswap juggling patterns are defined as permutations $f \colon \Z \to \Z$
with $f(t) \ge t$ for each $t \in \Z$, the interpretation being that the pattern is eternal,
and
a particle thrown at time $t$ is next thrown at a future time $f(t)$. The number $f(t)-t$ is
called the {\em throw} at time $t$, and the interpretation of a zero-throw is that the juggler
has nothing to throw and waits. The name siteswap stems from the fact that any pattern can be
generated from a constant-throw pattern by transpositions, i.e., by swapping the ``destination sites''
of two particles.

Ehrenborg and Readdy \cite{Ehrenborg_Readdy_1996} pointed out a connection between affine Weyl groups and periodic
siteswap patterns. Knutson, Lam and Speyer \cite{Knutson_Lam_Speyer_2011} have recently continued along a related line,
indexing positroid varieties in the Grassmannian by periodic siteswap patterns.
For recent developments on the combinatorics of periodic siteswap patterns
see e.g. \cite{Butler_Graham_2010, Chung_et_al_2010, Chung_Graham_2008, Hyatt_2013}.

Warrington \cite{Warrington_2005} deviated from the realm of deterministic 
patterns and calculated the stationary distribution for patterns with bounded, uniformly distributed, random throws.
Two of the authors \cite{Leskela_Varpanen_2012} continued
in an unbounded random setting, showing that a large class of random patterns is ergodic.
They also obtained an explicit stationary distribution for unbounded random
patterns with a geometric throw distribution. 

In this paper we derive an explicit stationary distribution for
random patterns with bounded throws distributed geometrically with a parameter $q \in (0,1)$.
Our state space consists of sets $B \subset \{0,\ldots,m-1\} \subset \Z_+$ that
are related to siteswaps $f \colon \Z \to \Z$ via
\[
B = B_t = \{x-t \ \vert \  x \ge t \text{ and } f^{-1}(x) < t\}.
\]
That is, the set $B$ consists of future throwing times (or destination sites) of the particles relative to the moment just before the current
time $t$. We say that the states $B$ contain {\em heights} of the particles.

We show that the stationary probability of having particles at heights $B \subset \Z_+$
is proportional to
\begin{equation}
\label{eq:bdd_equilibrium}
\prod_{x \in B} [\ 1 + v_B(x)\ ]_q \ q^x,
\end{equation}
where
\[
v_B(x) = \bigl| \{x, \ldots, m-1\}\setminus B \bigr|
\]
is the number of vacant slots above $x$, and $[k]_q = (1-q^k)/(1-q)$ denotes the $q$-analogue of the integer $k$. This formula is a
natural interpolation of the corresponding stationary distributions $\prod_{x \in B} q^x$ of the unbounded
system \cite{Leskela_Varpanen_2012} and $\prod_{x \in B} \ \bigl( 1+ v_B(x) \bigr)$
of the bounded system with uniformly distributed throw heights \cite{Warrington_2005}. Note that
the product term in \eqref{eq:bdd_equilibrium} occurs as a $q$-rook polynomial coefficient;
see \eqref{eq:rook_connection} and \cite[eq.\ (21)]{Haglund_1998}. Indeed, the key
ingredient in our proof is to extend the juggling dynamics and view the extended states as rook
placements on staircase Ferrers boards.

We also show that the stationary distribution of the bounded geometric pattern
with large $m$ can be uniformly approximated with the Gibbs measure
of the unbounded system considered in
\cite{Leskela_Varpanen_2012}. The proof utilizes the ultrafast mixing
property of the unbounded geometric system along with a stochastic coupling
argument.

As an application we can easily obtain the following types of convergence
formulae for $q$-analogues and their related $q$-Stirling numbers.
For any $n \ge 1$,
\[
\lim_{m\to\infty}
Z^{-1}\bigl[ m - n + 1 \bigr]_q^n
= (q;q)_n,
\]
where $(q;q)_n = \prod_{k=1}^n(1-q^{k})$ is the $q$-Pochhammer symbol, and
\begin{equation}
\label{Z}
Z = Z(m,n,q) = q^{-n + {m+1 \choose 2}}S_{1/q}[m+1,m-n+1].
\end{equation}
Here $S_q[a,b]$ denotes the $q$-Stirling number of the second kind,
defined by the recursion
\begin{equation}
\label{qstir}
S_q[a+1,b] = q^{b-1}S_q[a,b-1] + [b]_qS_q[a,b] \qquad ( 0 \le b \le a)
\end{equation}
with the initial conditions $S_q[0,0] = 1$ and $S_q[a,b] = 0$ for $b < 0$ or
$b > a$. We note that there are many $q$-Stirling numbers and \eqref{qstir} is
only one of several definitions.

Further by letting $m,n \to \infty$ so
that $m - n \to \infty$, we obtain
\[
\lim_{m,n\to\infty}
Z^{-1}\bigl[ m - n + 1 \bigr]_q^n
= \phi(q),
\]
where $\phi(q) = \prod_{k=1}^\infty(1-q^k)$ is the Euler function and
$Z$ is as in \eqref{Z} above.

Ayyer, Bouttier, Corteel and Nunzi \cite{ABCN} have recently extended several results of this
paper.
Probabilistic formulae involving $q$-combinatorics have also been found in the
context of birth processes related to the number of sources and paths in
directed random graphs \cite{Crippa_Simon_1997} as well as approximate counting \cite{Louchard_Prodinger_2008}. Finally, $q$-Stirling numbers of the second kind
have also appeared in connection with periodic siteswap patterns \cite{Ehrenborg_Readdy_1996}.

\section{Bounded juggler's exclusion process}
\label{sec:combinatorics}

A generic model for random juggling patterns called the \emph{juggler's exclusion
process} (JEP) was recently introduced by Leskelä and Varpanen in~\cite{Leskela_Varpanen_2012}.
The bounded JEP with $n$ particles and $m \ge n$ admissible particle heights $\{0, \ldots, m-1\}$ is a
random sequence $(X_0,X_1,\dots)$ of $n$-element subsets of $\{0, \ldots, m-1\}$ such that:
\begin{itemize}
  \item $X_{t+1} = X_t - 1$ when $0 \notin X_t$. (All particles fall down by one position when the hand is empty).
  \item $X_{t+1} = X_t^* \cup \{ \eta_t \}$ when $0 \in X_t$, where $X_t^* = (X_t \setminus 0) - 1$ and $\eta_t$ is a
  random integer in $\{0, \ldots, m-1\} \setminus X_t^*$. (The particle at hand is thrown into height $\eta_t$, not allowing collisions with the other particles.)
\end{itemize}

Analytically, the most interesting special case of the bounded model
with $n$ particles and $m \ge n$ admissible heights is the \emph{bounded geometric JEP}, where the
throw heights $\eta_t$ are defined as follows. Denote by $\ell = m-n+1$
the number of vacant throw heights, and let $(\xi_1, \xi_2, \ldots)$ be a sequence of independent
random integers on $\{0 \ldots \ell-1\}$ each having an $\ell$-truncated geometric distribution with
parameter $q \in (0,1)$ so that
\begin{equation}
 \label{eq:Truncated}
 \pr( \xi_t = x )
 \ = \ \left( \frac{1-q}{1-q^{\ell}} \right) q^x,
 \quad x \in \{0, \ldots \ell-1\}.
\end{equation}
Let $A \subset \Z_+$ be finite and denote by $\theta_A$ the order-preserving bijection from $\Z_+$
onto $\Z_+ \setminus A$. Then the random throw heights $\eta_t$ in the bounded geometric JEP are
given by
\[
 \eta_t = \theta_{X_t^*}(\xi_t).
\]
By \cite[Thm. 2.1]{Leskela_Varpanen_2012} we know that 
the bounded geometric JEP with $n$ particles and $m$ heights
is ergodic and characterized by a unique
stationary probability distribution. We denote the 
stationary probabilities by $\pi_{m,n,q}(B)$,
$B \in \Z_+^{(n)} = \{A \subset \Z_+: |A|=n\}$.

%\newpage
\subsection{A combinatorial formula for the stationary distribution}
\label{sec:bounded_equilibrium}
The probability transition matrix $P$ of an $n$-particle JEP
acts on probability distributions $\mu$ on $\Z_+^{(n)}$ according to $\mu \mapsto \mu P$
as follows. Denote $B = \{i_1,\ldots,i_n\} \in \Z_+^{(n)}$ and
$B+1 = \{i_1 + 1,\ldots, i_n+1\}$. At each step the juggler either throws one of the particles or does nothing.
Therefore
\[
\mu P(B) = \mu(B+1) + \sum_{k=1}^n \mu\bigl(\{0\} \cup (B \setminus \{i_k\}+1)\bigr)h^{B\setminus \{i_k\}}(i_k),
\]
where $h^{B\setminus\{i_k\}}(i_k)$ is the probability of throwing the particle $i_k \in B$ while the
other particles shift down from $B+1$ to $B$. Hence the problem of finding a stationary distribution amounts to finding
a probability distribution $\pi$ on $\Z_+^{(n)}$ such that
\begin{equation}
\label{eq:general_balance}
\pi(B) = \pi(B+1) + \sum_{k=1}^n \pi\bigl(\{0\} \cup (B \setminus \{i_k\}+1)\bigr)h^{B\setminus \{i_k\}}(i_k).
\end{equation}
In \cite{Leskela_Varpanen_2012} a solution was found in the one-particle
general case as well as in the unbounded geometric case, and in
\cite{Warrington_2005} a solution was found in the bounded case with
uniform throws. Solving the equation in the general unbounded case seems
hard, if not impossible. Obtaining a closed-form solution in the general
bounded case seems difficult.

\begin{theorem}
\label{the:bounded_equilibrium}
In the bounded geometric JEP the stationary distribution is
\begin{equation}
\label{eq:bounded_equilibrium}
\pi_{m,n,q}(B) = Z^{-1} \prod_{x \in B}\bigl[\ 1 + v_B(x) \ \bigr]_q \ q^x,
\end{equation}
where
$
v_B(x) = \bigl| \{x, \ldots, m-1\}\setminus B \bigr|
$
and $Z = Z(m,n,q)$ is as in \eqref{Z}.
\end{theorem}

Theorem \ref{the:bounded_equilibrium} readily allows to derive analytical formulas for various interesting characteristics of the
system in steady state. For example, the stationary probabilities of the ground state $\{0,\ldots,n-1\}$ and the highest-energy state
$\{m-n, \ldots, m-1\}$ are
\[
Z^{-1} [m-n+1]_q^n q^{n \choose 2} \ \text{ and } \ Z^{-1}q^{nm-{n+1\choose 2}},
\]
respectively. The probability that
zero is occupied, i.e., the long-run fraction of time slots at which the juggler performs a throw, is 
\[
  \frac{Z(m-1,n-1,q)}{Z(m,n,q)} [m-n+1]_q.
\]

We prove Theorem \ref{the:bounded_equilibrium} by considering an extended process where
the states are staircase Ferrers boards with $m+1$ columns and with $n$ non-attacking rooks
and where the dynamics is best introduced by an example.
Note that in the literature \cite{Garsia_Remmel_1986,Stanley_1986}
a staircase Ferrers board contains a void column; thus a board with $m+1$ columns refers to a board with highest column length $m$. In the example below we have $m=6$ (board with $7$ columns)
and $n=3$:
\begin{figure}[h!tbp]
\label{fig1}
\begin{center}
 \includegraphics{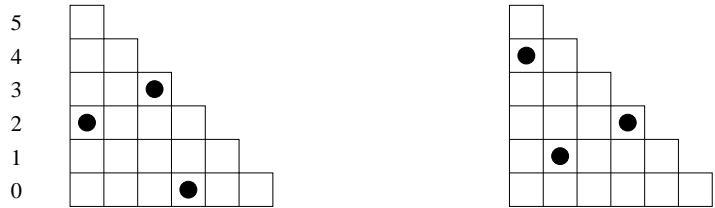}
\end{center}
\caption{An extended state and one of its followers.}
\end{figure}
Figure 1 depicts one extension of the situation where the non-extended state $\{0,2,3\}$ is followed by the
non-extended state $\{1,2,4\}$.
In the extended dynamics the particles drift downwards in a southeast direction, and the particles are
thrown from the bottom row to the leftmost column. The idea
 is to keep track not only of the remaining flight time (vertical axis) but also of the elapsed flight time
(horizontal axis) of the particles. The ``non-attacking rook'' model is
clearly implied by the exclusion rule that no two particles may collide.
The concept of an extended state already appeared in
\cite{Warrington_2005}, albeit in a different form.

\begin{proposition}
The extended process has a unique stationary distribution.
\end{proposition}
\begin{proof}
It suffices to prove that the extended process is irreducible and aperiodic. This is done as in \cite[Lemma 2.1]{Leskela_Varpanen_2012}.
We define the extended ground state as having the $n$ particles in the $n$-diagonal of the staircase, as depicted in Figure 2.
The extended ground state can be reached from any extended state, including itself, by always throwing to the lowest available height. 
It is also easy to see that any extended state can be reached from the extended ground state by a suitable sequence of throws.
Aperiodicity follows from the fact that there is a strictly positive probability to remain in the extended ground state.
\begin{figure}[h!tbp]
\label{fig:ext_ground}
\begin{center}
\includegraphics{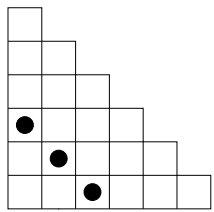}
\end{center}
\caption{The extended ground state for $m=6$, $n=3$.}
\end{figure}
\end{proof}

We next introduce a $q$-counting statistic $\text{circ}(C)$ for a configuration $C$ of non-attacking rooks
on a staircase Ferrers board. The statistic is also best introduced by an example
($m=6$, $n=3$):
\begin{figure}[h!tbp]
\label{fig2}
\begin{center}
 \includegraphics{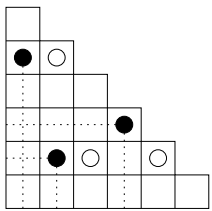}
\end{center}
\caption{The circ statistic.}
\end{figure}
Given a non-attacking rook configuration $C$, we disable all positions which occur in the same row and column of a rook which are directly to the left or directly below a rook, and place circles to the remaining positions in those rows that have a rook.
Then $\text{circ}(C)$ counts the total number of circles in $C$. In Figure 3 we have $\text{circ}(C) = 3$. We note that the statistics $\text{circ}(C)$
is not really new. It is a modified version of the statistic appearing in \cite{Garsia_Remmel_1986}, where circles are allowed also in
unoccupied squares in rows with no rooks.

\begin{definition}
In what follows, we denote by $C_n(S_{m+1})$ the set of all configurations of $n$ non-attacking rooks on a staircase Ferrers board $S_{m+1}$
with $m+1$ columns.
\end{definition}

\begin{lemma}
\label{lem:circ_sum}
\begin{equation*}
\sum_{C \in C_n(S_{m+1})} q^{\rm{circ}(C)} = G_q[m+1,m-n+1],
\end{equation*}
where the numbers $G_q[a,b]$ are defined by the recursion formula
\begin{equation}
\label{eq:Gould_def}
G_q[a+1,b] = G_q[a,b-1] + [b]_qG_q[a,b],
\end{equation}
where $0 \le b \le a$, $G_q[0,0]=1$ and $G_q[a,b] = 0$ for $b < 0$ or $b > a$.
\end{lemma}
\begin{proof}
The proof is identical to that of \cite[Thm. 1.1]{Garsia_Remmel_1986}.
A configuration in $C_n(S_{m+1})$ is obtained by placing either all the $n$ rooks in rows other
than the bottom row in $S_{m+1}$ or by placing one of the rooks in the bottom row. In the first case there are
no additional circles compared to the corresponding configuration in $C_n(S_m)$.

In the second case, there are $m-n+1$ possible locations in the bottom row to place a new rook. If the new rook is placed in the $i$th
available square, counting from right to left,
there are $i-1$ additional circles compared to the corresponding configuration in $C_{n-1}(S_m)$.
Denoting
\[
K_q[n,m+1] = \sum_{C \in C_n(S_{m+1})} q^{\text{circ}(C)}
\]
and summing over all configurations in $C_n(S_{m+1})$ we obtain
\begin{equation}
\label{eq:K_recursion}
K_q[n,m+1] = K_q[n,m] + [m-n+1]_qK_q[n-1,m],
\end{equation}
because there is a total of $m-n+1$ available squares in the second case considered above. Thus
\[
G_q[m+1,m+1-n] = K_q[n,m+1]
\]
by comparing \eqref{eq:K_recursion} with \eqref{eq:Gould_def}.
\end{proof}

The numbers $G_q[a,b]$ are called Gould $q$-Stirling numbers of the second kind
 \cite{Gould_1961} or modified $q$-Stirling numbers of the second kind
\cite{Ehrenborg_2003}. They are related
to the numbers $S_q[a,b]$ by
\begin{equation}
\label{eq:Gould}
 S_q[a,b] = q^{\left(\begin{smallmatrix}b\\2\end{smallmatrix}\right)}G_q[a,b].
\end{equation}

Next consider the stationarity equation for the extended process. Again the juggler either waits or throws one of the particles, so the equation is similar to \eqref{eq:general_balance}. We write it as
\begin{equation}
\label{eq:ext_equilibrium}
\mu(C) = \mu(\widehat{C}) + \sum_{C^*} \mu(C^*)h(C^*,C),
\end{equation}
where $\widehat{C} \in C_n(S_{m+1})$ denotes the unique predecessor of $C \in C_n(S_{m+1})$ when the juggler has
moved from state $\widehat{C}$ to state $C$ by waiting,
$C^* \in C_n(S_{m+1})$ denotes a predecessor of $C$ when the juggler has moved
from state $C^*$ to state $C$ by a throw with probability $h(C^*,C)$, and the sum is taken over all
predecessors of $C$. Note that
$\widehat{C}$ is void if there is a particle in the leftmost column of $C$, and that the
sum is empty if there is no particle in the leftmost column of $C$.
\begin{lemma}
\label{lem:temp}
The equation \eqref{eq:ext_equilibrium} is solved by
\begin{equation}
\label{eq:ext_solution}
\mu(C) = \frac{q^{-\rm{circ}(C)}}{G_{1/q}[m+1,m-n+1]}.
\end{equation}
\end{lemma}
\begin{proof}
Clearly $\text{circ}(\widehat{C}) = \text{circ}(C)$ when the juggler has waited. If the juggler has thrown
a particle to the $j$th available square, counting from the top down in the leftmost column of $C$, the
throw probability has been $h(C^*,C) = q^{-(j-1)}/[m-n+1]_{1/q}$. Moreover,
removing the particle from the leftmost column of $C$ decreases the
circle count by $j-1$.
If we similarly remove a particle from the $i$th non-attacking square (counting from right to left)
of the bottom row in $C^*$, we decrease the circle count by $i-1$. Summing over the $m-n+1$ predecessors
of $C$ and keeping track of the circle count we obtain
\[
q^{-\rm{circ(C)}} = \sum_{i=1}^{m-n+1}q^{-{\rm circ}(C)-(i-1)+(j-1)}\frac{q^{-(j-1)}}{[m-n+1]_{1/q}}.
\]
Hence the equation \eqref{eq:ext_equilibrium} is solved by $\widetilde{\mu}(C) = q^{-\rm{circ}(C)}$,
and normalization yields the denominator $G_{1/q}[m+1,m-n+1]$ by Lemma
\ref{lem:circ_sum}.
\end{proof}

\begin{lemma}
\label{lem:extensions}
Let $B \subset \{0, \ldots, m-1\}$ with $\abs{B} = n$. Then the following identity holds
\begin{equation}
\label{eq:ext_sum}
 \sum_C q^{-\rm{circ}(C)} = \prod_{x \in B}\bigl[\ 1+\bigl| \{x, \ldots, m-1\} \setminus B \bigr| \ \bigr]_{1/q},
\end{equation}
where the sum is taken over all possible extensions $C \in C_n(S_{m+1})$ of $B$.
\end{lemma}
\begin{proof}
We proceed by induction on $n$. If $n=1$, there are $m-x$ possible extensions
$C$ of $B = \{x\} \subset [0,m-1]$, and the sum \eqref{eq:ext_sum} is $1+q^{-1}+\cdots+q^{-m+1+x} = [m-x]_{1/q}$.
When $n \ge 2$ consider $x = \min B$. There are $m-x-(n-1)$ ways of placing a non-attacking rook
to row $x \in S_{m+1}$ after $n-1$ non-attacking rooks have already been placed to higher rows. The claim follows.
\end{proof}

\begin{proof}[Proof of Theorem \ref{the:bounded_equilibrium}]
By Lemmas \ref{lem:temp} and \ref{lem:extensions}  we have
\[
\pi_{m,n,q}(B) = \frac{ \prod_{x \in B}\bigl[\ 1+\bigl| \{x, \ldots, m-1\}\setminus B \bigr| \ \bigr]_{1/q} }{G_{1/q}[m+1,m-n+1]},
\]
i.e., by \eqref{eq:Gould},
\[
\pi_{m,n,q}(B) = \frac{\prod_{x \in B}\bigl[\ 1+\bigl| \{x, \ldots, m-1\}\setminus B \bigr| \ \bigr]_{1/q} }{q^{m-n+1 \choose 2}S_{1/q}[m+1,m-n+1]}.
\]
Denoting $B=\{i_1, \dots i_n\}$ with $i_1 < \cdots < i_n$, we have
\begin{equation}
\label{eq:rook_connection}
\prod_{x \in B}\bigl[\ 1 + \bigl| \{x, \ldots, m-1\}\setminus B \bigr|\ \bigr]_{1/q}
= \prod_{k=1}^n[m-n-i_k+k]_{1/q}.
\end{equation}
Using \eqref{eq:rook_connection} along with the fact that
$[k]_q = q^{k-1}[k]_{1/q}$ for any $k$, we obtain
\[
\pi_{m,n,q}(B) = \frac{\prod_{x \in B}\bigl[\ 1 + \bigl| \{x, \ldots, m-1\}\setminus B \bigr| \ \bigr]_q \ q^x}{q^{n+{m + 1 \choose 2}}S_{1/q}[m+1,m-n+1]}
\]
as claimed.
\end{proof}

\section{Convergence}
\label{sec:Convergence}
\subsection{Total variation distance and coupling}
We recall some definitions from probability that are used in the sequel.
A probability distribution on a countable set $S$ is a function $\mu: S \to \R_+$ such that $\sum_{A \in S} \mu(A) = 1$.
The total variation distance between two probability
distributions $\mu$ and $\nu$ on a countable set $S$ is defined by
\[
 \normtv{\mu - \nu}
 = \frac{1}{2} \sum_{ A \in S } | \mu(A) - \nu(A) |.
\]
A \emph{coupling} of $\mu$ and $\nu$ is a random vector $(X,Y)$ with values in $S \times S$ such
that $\pr( X = A ) = \mu(A)$ and $\pr( Y = B ) = \nu(B)$ for all $A,B \in S$. A coupling $(X,Y)$
of $\mu$ and $\nu$ is \emph{maximal} if
\[
 \pr( X \neq Y ) \ = \ \normtv{\mu - \nu}.
\]
A maximal coupling of $\mu$ and $\nu$ always exists (e.g.\
\cite[Prop. 4.7]{Levin_Peres_Wilmer_2008}).% or \cite[Thm. 1.4.2]{Thorisson_2000}).

\subsection{Bounded geometric JEP with many vacant heights}

Consider the bounded geometric JEP with $n$ particles, $m \ge n$ admissible heights, and a throw
height parameter $q \in (0,1)$. We will show that when the number of vacant throw heights $\ell =
m - n +1$ is large, the stationary probabilities of the bounded geometric JEP can be uniformly
approximated by the stationary distribution of the unbounded geometric JEP considered in
\cite{Leskela_Varpanen_2012}. The unbounded geometric JEP is defined by replacing the $\ell$-truncated geometric
distribution in \eqref{eq:Truncated} by a standard geometric random distribution so that
\[
 \pr( \xi_t = x ) = (1-q) q^x, \quad x \in \Z_+.
\]
We denote the transition probability matrix of the bounded geometric JEP by $P_{m,n,q}$, and we
view its stationary distribution $\pi_{m,n,q}$ as a probability distribution on $\Z_+^{(n)}$, although all its
mass is concentrated on $\{0, \ldots, m-1\}^{(n)}$. We denote by
$P_{\infty,n,q}$ the transition matrix of the unbounded $q$-geometric JEP with $n$ particles, and
by $\pi_{\infty,n,q}$ its stationary distribution given by \cite[Thm. 3.3]{Leskela_Varpanen_2012} 
\begin{equation}
\label{eq:LV_3.3}
 \pi_{\infty,n,q}(B) = \frac{(q;q)_n}{q^{\binom{n}{2}}} \prod_{x \in B} q^x,
 \quad B \in \Z_+^{(n)}.
\end{equation}

To state the approximation result in its most general form, we consider a sequence of bounded
geometric JEPs indexed by $k =1,2,\dots$ where the $k$-th process has $n(k)$ particles, $m(k) \ge
n(k)$ admissible heights, and a throw height parameter $q(k) \in (0,1)$. For the reader's
convenience we shall omit the symbol $k$ in what follows.

\begin{theorem}
\label{the:vjep_conv}
Assume that the number of vacant throw heights $\ell = m-n+1$ in the sequence of bounded geometric
JEPs satisfies
\begin{equation}
 \label{eq:mGrowth}
 \ell \log q^{-1} - \log m \to \infty
 \qquad \text{as $k \to \infty$}.
\end{equation}
Then the total variation distance between the stationary distributions of the bounded JEP and its unbounded variant
satisfies
\[
 \normtv{ \pi_{m,n,q} - \pi_{\infty,n,q} } \to 0
 \qquad \text{as $k \to \infty$}.
\]
\end{theorem}
We first prove the following result.
\begin{lemma}
\label{the:Coupling}
Consider two different $n$-particle JEPs on $\Z_+^{(n)}$, one with
throw height distribution
$\nu$ and the other with throw height distribution $\widehat\nu$, and denote by $\mu$ and $\widehat\mu$ their
corresponding initial distributions. Then
\[
 \normtv{ \mu P^t  - \widehat\mu \widehat P^t }
 \ \le \
 1 - (1- \normtv{ \mu-\widehat\mu })(1-\normtv{\nu-\widehat\nu})^t
\]
for all $t\ge 0$.
\end{lemma}
\begin{proof}
Let $(X_0,\widehat X_0)$ be a maximal coupling of $\mu$ and $\mu'$, so that
$\normtv{\mu - \widehat \mu} =
\pr(X_0 \neq \widehat X_0)$. Similarly, let $(\xi,\widehat\xi)$ be a maximal coupling of $\nu$ and
$\widehat\nu$. Let $((\xi_1,\widehat\xi_1),(\xi_2,\widehat\xi_2), \dots)$ be an independent sequence of copies
of $(\xi,\widehat\xi)$, which is also independent of $(X_0,\widehat X_0)$.

Generate a path $(X_s)_{s \ge 0}$ of the first JEP using the initial configuration $X_0$ and throw heights $\xi_1,\xi_2,\dots$, and a path $(\widehat X_s)_{s \ge 0}$ of the second JEP from initial configuration $\widehat X_0$ using throw heights $\widehat \xi_1, \widehat \xi_2, \dots$. Because the pair $(X_t, \widehat X_t)$ is a coupling of $\mu P^t$ and $\widehat\mu \widehat P^t$, it follows that \cite[Sec. 1.5.4]{Thorisson_2000}
\[
 \normtv{ \mu P^t  - \widehat\mu \widehat P^t }
 \le \pr( X_t \neq \widehat X_t ).
\]
Because $X_t = \widehat X_t$ on the event $\{X_0 = \widehat X_0\} \cap
\left({\displaystyle \bigcap_{s=1}^t \{\xi_s = \widehat\xi_s\} }\right)$, it follows that
\[
 \pr( X_t \neq \widehat X_t )
 \le 1 - \pr(X_0 = \widehat X_0) \pr( \xi = \widehat\xi )^t.
\]
The claim follows after combining the above inequalities.
\end{proof}
\begin{proof}[Proof of Theorem \ref{the:vjep_conv}]
A key ingredient of the proof is to note
that the unbounded geometric JEP reaches its stationary distribution \emph{exactly} after all particles have been thrown \cite[Thm. 3.2]{Leskela_Varpanen_2012}. Because $\pi_{m,n,q}$ is supported on $\{0, \ldots, m-1\}^n$, all particles in the unbounded geometric JEP with initial distribution $\pi_{m,n,q}$ have been thrown by time $m$. As a consequence, $\pi_{m,n,q} P^t_{\infty,n,q} = \pi_{\infty,n,q}$
for all $t \ge m$. Because $\pi_{m,n,q}$ is invariant for $P^t_{m,n,q}$, this implies that
\begin{equation}
 \label{eq:KeyUnbounded}
 \pi_{m,n,q} - \pi_{\infty,n,q}
 \ = \
 \pi_{m,n,q} P^t_{m,n,q} - \pi_{m,n,q} P^t_{\infty,n,q}
\end{equation}
for $t \ge m$.

Because the total variation distance between the standard and the $\ell$-truncated geometric distribution equals $q^\ell$, Lemma~\ref{the:Coupling} shows that
\[
 || \pi_{m,n,q} P_{m,n,q}^t  - \pi_{m,n,q} P_{\infty,n,q}^t ||
 \ \le \
 1 - (1-q^\ell)^t.
\]
By substituting $t = m$ into~\eqref{eq:KeyUnbounded}, it thus follows that
\[
 || \pi_{m,n,q} - \pi_{\infty,n,q} ||
 \ \le \
 1 - (1-q^\ell)^m
 \ \le \
 m q^\ell,
\]
where the right side tends to zero due to assumption~\eqref{eq:mGrowth}.
\end{proof}

\subsection{Combinatorial limit formulae}
\label{sec:jep_conv}
\begin{corollary}
\label{cor:fixed_n}
For any $n \in \Z_+$,
\begin{equation}
\label{eq:fixed_n}
\lim_{m\to\infty}
Z^{-1}\bigl[ m - n + 1 \bigr]_q^n
= (q;q)_n,
\end{equation}
where $Z$ is as in \eqref{Z}.
\end{corollary}
\begin{proof}
For the ground state $B = \{0, \ldots, n-1\}$, the stationary distribution
\eqref{eq:LV_3.3} for the unbounded process
reduces to $\pi_{\infty,n,q}(B) = (q;q)_n$, and by
\eqref{eq:bdd_equilibrium} we have
$\pi_{m,n,q}(B) = Z^{-1}\bigl[m-n+1\bigr]_q^n$.
The claim follows from Theorem \ref{the:vjep_conv}.
\end{proof}

\begin{corollary}
\label{cor:variable_n}
When both $n \in \Z_+$ and $m \in \Z_+$ approach infinity such that
\eqref{eq:mGrowth} holds (especially such that $m-n \to \infty$), we have
\begin{equation}
\label{eq:virtual_convergence}
\lim_{m,n\to\infty}
Z^{-1}\bigl[ m - n + 1 \bigr]_q^n
 = \phi(q),
%\end{aligned}
\end{equation}
where $Z$ is as in \eqref{Z}.
\end{corollary}
\begin{proof}
For any $q \in (0,1)$, the $q$-Pochhammer symbol $(q;q)_n$ converges to
the Euler function $\phi(q)$ as $n \to \infty$. The claim then follows as
in the proof of Corollary \ref{cor:fixed_n}.
\end{proof}

\begin{remark}
\label{rmk:final}
Corollary \ref{cor:fixed_n} holds in a more general form, because
one could have any $B \in \Z_+^{(n)}$ in the proof and still
obtain stationary distributions for random juggling processes in both bounded and
unbounded settings. However, in Corollary \ref{cor:variable_n}
the right-hand side $\phi(q)$ is not a proper stationary distribution for a JEP,
because the unbounded process fills $\Z_+$ when the number of
particles approaches infinity. We will address this
issue more carefully in a subsequent paper. It turns out that
$\phi(q)$ is the stationary distribution of the ground state in a {\it virtual} setting considered
in \cite{Knutson_Lam_Speyer_2011}.
\end{remark}

\section*{Acknowledgements}
We thank Allen Knutson and Greg Warrington for inspiration and the two anonymous referees for their thorough reports.

\end{document}